\documentclass[11pt]{amsart}
\usepackage{amssymb,amsmath,mathrsfs}
\usepackage{enumerate,tikz}
\usepackage{hyperref}

\usetikzlibrary{patterns}
\usetikzlibrary{decorations.pathreplacing, calligraphy}
\usetikzlibrary{decorations.markings}

\newtheorem{thm}{Theorem}[section]
\newtheorem{prop}[thm]{Proposition}
\newtheorem{cor}[thm]{Corollary}
\newtheorem{lem}[thm]{Lemma}

\numberwithin{equation}{section}

\def\Gr{\mathscr{G}}
\def\V{\mathscr{V}}
\def\W{\mathscr{W}}
\def\P{\mathscr{P}}
\def\A{\mathcal{A}}
\def\B{\mathcal{B}}
\def\E{\mathcal{E}}

\def\abs#1{\lvert#1\rvert}
\DeclareMathOperator{\des}{des}
\DeclareMathOperator{\inv}{inv}
\DeclareMathOperator{\id}{id}

\def\R{\rule[-1ex]{0ex}{3.6ex}}
\def\arr{\arrow[line width=0.3mm]{stealth}}

\tikzstyle{mesh}=[pattern=north east lines, pattern color=gray!70, draw=gray]
\tikzstyle{meshPatt}=[pattern=north west lines, pattern color=black!30!green, draw=gray]
\tikzstyle{meshMost}=[pattern=crosshatch, pattern color=black!10!orange!70, draw=gray]

\newcommand\dyckpath[4]{
\def\diam{0.1}
\def\sf{{#4}}
  \begin{tikzpicture}[scale=\sf]
    \draw[help lines] (#1) -- ++(#2*2,0);
    \draw[line width=1pt] (#1) foreach \dir in {#3}{ -- ++(\dir*90-45:1.41)};
    \draw[fill] (#1) circle (\diam);
    \draw[fill] (#1) foreach \dir in {#3}{ ++(\dir*90-45:1.41) circle (\diam)};
  \end{tikzpicture}
}

\begin{document}
\title{Pattern-avoiding even and odd Grassmannian permutations}
\author[Juan Gil]{Juan B. Gil}
\address{Penn State Altoona\\ 3000 Ivyside Park\\ Altoona, PA 16601, U.S.A.}
\email{jgil@psu.edu}

\author[Jessica Tomasko]{Jessica A. Tomasko}
\address{Penn State Altoona\\ 3000 Ivyside Park\\ Altoona, PA 16601, U.S.A.}
\email{jat5880@psu.edu}

\begin{abstract}
In this paper, we investigate pattern avoidance of parity restricted (even or odd) Grassmannian permutations for patterns of sizes 3 and 4. We use a combination of direct counting and bijective techniques to provide recurrence relations, closed formulas, and generating functions for their corresponding enumerating sequences. In addition, we establish some connections to Dyck paths, directed multigraphs, weak compositions, and certain integer partitions.
\end{abstract}

\maketitle

%%%%%%%%%%%%%%%%%%%%%%%%%%%%%%%%%%%%%%%%%%%%
\section{Introduction}

A {\em Grassmannian permutation} is a permutation having at most one descent. Note that a permutation $\pi$ is Grassmannian if and only if its reverse complement $\pi^{rc}$ is Grassmannian. We refer to the book by Kitaev \cite{Kitaev11} for basic definitions of permutations and pattern avoidance.

We let $\Gr_n$ denote the set of Grassmannian permutations on $[n]=\{1,\dots,n\}$. This set is in bijection with the set of Dyck paths of semilength $n$ having at most one long ascent.\footnote{An explicit bijection can be found in \cite[Section~4]{GiTo22a}.} We will call them {\em Grassmannian Dyck paths}. Further, we let $\Gr_n(\sigma)$ denote the set of elements in $\Gr_n$ that avoid the pattern $\sigma$, and let $\Gr^*_{n}(\sigma)=\Gr_{n}(\sigma)\backslash\{\id_n\}$. The enumeration of $\Gr_n(\sigma)$ for a pattern $\sigma$ of arbitrary size was studied in \cite{GiTo22a}. Clearly, $\Gr_n(\sigma)=\Gr_n$ whenever $\des(\sigma)>1$. If $\des(\sigma)=1$, we have the following general result.

\begin{thm}[{\cite[Thm.~3.1]{GiTo22a}}]\label{thm:GrassPerms}
If $k\ge 3$ and $\sigma\in S_k$ with $\des(\sigma)=1$, then
\[ \abs{\Gr_n(\sigma)}  = 1 + \sum_{j=3}^k\binom{n}{j-1} \text{ for } n\in\mathbb{N}. \]
\end{thm}

In this paper, we will study the subset of pattern-avoiding Grassmannian permutations with an additional parity restriction. Recall that a permutation is said to be {\em even} if it has an even number of inversions (occurrences of the pattern 21); otherwise the permutation is said to be {\em odd}. As discussed in \cite[Remark~5.2]{GiTo22a},
\[ \abs{\Gr^{even}_n} = 2^{n-1} + 2^{\lfloor \frac{n-1}{2}\rfloor} - n \;\text{ and }\; \abs{\Gr^{odd}_n} = 2^{n-1} - 2^{\lfloor \frac{n-1}{2}\rfloor}. \]

\begin{prop}[{\cite[Prop.~5.1]{GiTo22a}}]
The set $\Gr^{odd}_n$ is in bijection to the set of Grassmannian Dyck paths of semilength $n$ having an odd number of peaks at even height. On the other hand, the elements of $\Gr^{even}_n$ correspond to Grassmannian Dyck paths with an even number of peaks at even height.
\end{prop}

As it turns out, the enumeration of  $\Gr^{odd}_n(\sigma)$ (or $\Gr^{even}_n(\sigma)$) leads to multiple Wilf equivalence classes. Because of Theorem~\ref{thm:GrassPerms}, it suffices to only enumerate the odd or the even pattern-avoiding Grassmannian permutations. We will focus on odd permutations and will prove that if $\des(\sigma)=1$, there are two equivalence classes for $\Gr^{odd}_n(\sigma)$ when $|\sigma|=3$ (Table~\ref{tab:size3patt}) and five equivalence classes when $|\sigma|=4$ (Table~\ref{tab:size4patt}). For patterns of sizes 5, 6, and 7, there appear to be three, seven, and four Wilf equivalence classes, respectively. 

\begin{table}[ht]
\small
\def\R{\rule[-1.6ex]{0ex}{4.4ex}}
\begin{tabular}{|c|l|c|c|c|} \hline
\R Pattern $\sigma$ & \hspace{45pt} $|\Gr^{odd}_n(\sigma)|$ & Gen.\ function & OEIS \\[2pt] \hline
\R 123 & 0, 1, 2, 1, 0, 0, \dots && \\ \hline
\rule[-2.6ex]{0ex}{6.4ex} \parbox[m]{4.5ex}{132\\ 213} & 0, 1, 1, 3, 3, 6, 6, 10, 10, 15, 15, 21, \dots & $\frac{x^2}{(1-x)^3(1+x)^2}$ & A008805 \\ \hline
\rule[-2.6ex]{0ex}{6.4ex} \parbox[m]{4,5ex}{231\\ 312} & 0, 1, 2, 4, 6, 9, 12, 16, 20, 25, 30, 36, \dots & $\frac{x^2}{(1-x)^3(1+x)}$ & A002620 \\ \hline\hline
\R & \hspace{45pt} $|\Gr^{even}_n(\sigma)|$ & & \\[2pt] \hline
\R 123 & 1, 1, 2, 1, 0, 0, \dots && \\ \hline
\rule[-2.6ex]{0ex}{6.4ex} \parbox[m]{4.5ex}{132\\ 213} & 1, 1, 3, 4, 8, 10, 16, 19, 27, 31, 41, 46, \dots & $\frac{x + x^4 + x^5}{(1-x)^3(1+x)^2}$ & A131355 \\ \hline
\rule[-2.6ex]{0ex}{6.4ex} \parbox[m]{4.5ex}{231\\ 312} & 1, 1, 2, 3, 5, 7, 10, 13, 17, 21, 26, 31, \dots & $\frac{x - x^2 + x^4}{(1-x)^3(1+x)}$ & A033638 \\ \hline
\end{tabular}
\bigskip
\caption{} \label{tab:size3patt}
\end{table}

\begin{table}[ht]
\small
\def\R{\rule[-1.6ex]{0ex}{4.6ex}}
\begin{tabular}{|c|l|c|c|c|} \hline
\R Pattern $\sigma$ & \hspace{45pt} $|\Gr^{odd}_n(\sigma)|$ & Gen.\ function & OEIS \\[2pt] \hline
\R 1234 & 0, 1, 2, 6, 4, 3, 0, 0, \dots && \\ \hline
\rule[-2.5ex]{0ex}{6.4ex} \parbox[m]{11.5ex}{1243, 2134\\ 2341, 4123} & 0, 1, 2, 5, 9, 16, 25, 38, 54, 75, \dots & $\frac{x^2(1+x^3)}{(1-x)^4(1+x)^2}$ & A175287 \\ \hline
\R 1324 & 0, 1, 2, 5, 8, 16, 20, 38, 40, 75, \dots & $\frac{x^2(1+2x+x^2+2x^4)}{(1-x)^4(1+x)^4}$ & A361270 \\ \hline
\R 1342, 3124 & 0, 1, 2, 6, 9, 19, 25, 44, 54, 85, \dots & $\frac{x^2(1+x+x^2+x^4)}{(1-x)^4(1+x)^3}$ & A361271 \\ \hline
\rule[-2.6ex]{0ex}{6.4ex} \parbox[m]{11.5ex}{1423, 2314\\ \hspace*{2ex} 3412} & 0, 1, 2, 6, 10, 19, 28, 44, 60, 85, \dots & $\frac{x^2(1+x^2)}{(1-x)^4(1+x)^2}$ & A005993 \\ \hline
\R 2413 & 0, 1, 2, 5, 8, 14, 20, 30, 40, 55, 70, \dots & $\frac{x^2}{(1-x)^4(1+x)^2}$ & A006918 \\ \hline\hline
\R & \hspace{45pt} $|\Gr^{even}_n(\sigma)|$ & & \\[2pt] \hline
\R 1234 & 1, 1, 3, 5, 6, 2, 0, 0, \dots && \\ \hline
\rule[-2.5ex]{0ex}{6.4ex} \parbox[m]{11.5ex}{1243, 2134\\ 2341, 4123} & 1, 1, 3, 6, 12, 20, 32, 47, 67, 91, \dots & $\frac{x-2x^2+2x^3+x^4-x^5}{(1-x)^4(1+x)}$ & A361272 \\ \hline
\R 1324 & 1, 1, 3, 6, 13, 20, 37, 47, 81, 91, \dots & $\frac{x+x^2-x^3+2x^4+7x^5+2x^6-x^7-x^8}{(1-x)^4(1+x)^4}$ & A361273 \\ \hline
\R 1342, 3124 & 1, 1, 3, 5, 12, 17, 32, 41, 67, 81, \dots & $\frac{x-x^3+2x^4+4x^5-x^6-x^7}{(1-x)^4(1+x)^3}$ & A361274 \\ \hline
\rule[-2.6ex]{0ex}{6.4ex} \parbox[m]{11.5ex}{1423, 2314\\ \hspace*{2ex} 3412} & 1, 1, 3, 5, 11, 17, 29, 41, 61, 81, \dots & $\frac{x-x^2+2x^4+x^5-x^6}{(1-x)^4(1+x)^2}$ & A361275 \\ \hline
\R 2413 & 1, 1, 3, 6, 13, 22, 37, 55, 81, 111, \dots & $\frac{x-x^2+3x^4+x^5-x^6}{(1-x)^4(1+x)^2}$ & A361276 \\ \hline
\end{tabular}
\bigskip
\caption{} \label{tab:size4patt}
\end{table}

The paper is structured as follows. In Section~\ref{sec:size3patts}, we warm up with the avoidance of patterns of size 3. We give closed formulas for the two nontrivial Wilf equivalence classes and give simple combinatorial interpretations in terms of Grassmannian Dyck paths. In Section~\ref{sec:size4patts}, we discuss patterns of size 4 with exactly one descent and use bijective methods to derive closed formulas and recurrence relations for the enumeration of the corresponding pattern-avoiding Grassmannian permutations. Finally, in Section~\ref{sec:bijections}, we conclude with  interesting combinatorial interpretations for $\Gr^{odd}_n(2413)$ and $\Gr^{odd}_n(3412)$. We give a bijection between the first set and the set of integer partitions of $n+2$ having Durfee square\footnote{Largest square that is contained within the partition's Ferrers diagram, see \cite{StanEC1}, \cite{SylvFran}.} of size 2. In addition, for $n\ge 2$, we provide a bijection between $\Gr^{odd}_n(3412)$ and the set of multidigraphs on two nodes having $n-2$ edges.

Regarding future research directions, we believe that the techniques used here and in our previous paper \cite{GiTo22a} can be applied to discuss parity restricted Grassmannian permutations avoiding patterns of larger sizes. Studying such permutations might shed some light into a possible unifying approach for their enumeration.

%%%%%%%%%%%%%%%%%%%%%%%%%%%%%%%%%%%%%%%%%%%%
\section{Patterns of size 3}
\label{sec:size3patts}

We start by noting that, by means of the reverse complement map, we have 
\[ \abs{\Gr^{odd}_{n}(132)} = \abs{\Gr^{odd}_{n}(213)} \;\text{ and }\; \abs{\Gr^{odd}_{n}(231)} = \abs{\Gr^{odd}_{n}(312)}. \]
We will provide closed formulas for $\abs{\Gr^{odd}_{n}(132)}$ and $\abs{\Gr^{odd}_{n}(312)}$, as well as interpretations for all four patterns in terms of Grassmannian Dyck paths. 

\begin{thm} \label{thm:132avoiders} %ref OEIS A008805
For $n\ge 2$ we have
\begin{equation*}
  \abs{\Gr^{odd}_{n}(132)} = \binom{\lfloor\frac{n}{2}\rfloor+1}{2}.
\end{equation*}
\end{thm}
\begin{proof}
Excluding the identity, any Grassmannian permutation that avoids $132$ must be of the form
\begin{equation*}
\underbrace{(i+1) \cdots m}_{\tau_{1}}\, \underbrace{\!\phantom{(}1\cdots i \phantom{)}\!}_{\tau_{2}}\, \underbrace{(m+1) \cdots n}_{\tau_{3}} \;\text{ with } i,m\in\{1,\dots,n\},\; i<m,
\end{equation*}
where $\tau_{3} = \varepsilon$ (the empty word) if $m = n$. Note that $m1$ forms the unique descent of the permutation. The permutation is odd if and only if the product $|\tau_1| \cdot |\tau_2|$ is odd. This implies $m$ is even and $i$ is odd.

If $n=2k$, we can construct the elements of $\Gr^{odd}_{n}(132)$ by placing $m \in \{2, 4, \dots,2k\}$ at position $d \in \{1,3,\dots,2k-1\}$, with the restriction that $m>d$. Thus, there are 
\[ k + (k-1) + \dots + 2 + 1 = \frac{k(k+1)}{2} = \binom{\frac{n}{2}+1}{2} \] 
even-sized such permutations. Alternatively, note that since $m$ is even and $i$ is odd, every permutation in $\Gr^{odd}_{n}(132)$ corresponds to a choice of two elements $i' < m'$ from $[\frac{n}{2} + 1]$, where $i = 2i'-1$ and $m = 2m'-2$.

If $n$ is odd, the entry $n$ cannot be part of the descent, so every element of $\Gr^{odd}_{n}(132)$ can be created by appending $n$ at the end of a corresponding permutation of size $n-1$. Thus,
\begin{equation*}
\abs{\Gr^{odd}_{n}(132)} =
\begin{cases}
\binom{\frac{n}{2}+1}{2} & \text{if $n$ is even,} \\[4pt]
\abs{\Gr^{odd}_{n-1}(132)} & \text{if $n$ is odd,}
\end{cases}
\end{equation*}
which gives the claimed formula.
\end{proof}

The following two propositions are immediate consequences of the bijections discussed in \cite[Prop.~4.1 \& Prop.~5.1]{GiTo22a}.

\begin{prop} \label{prop:132dyckpaths}
The elements of $\Gr^{odd}_{n}(132)$ are in one-to-one correspondence with the odd Grassmannian Dyck paths of semilength $n$ that begin with a long ascent and have exactly one long descent.
\end{prop}
\def\eps{0.25}
\begin{figure}[ht]
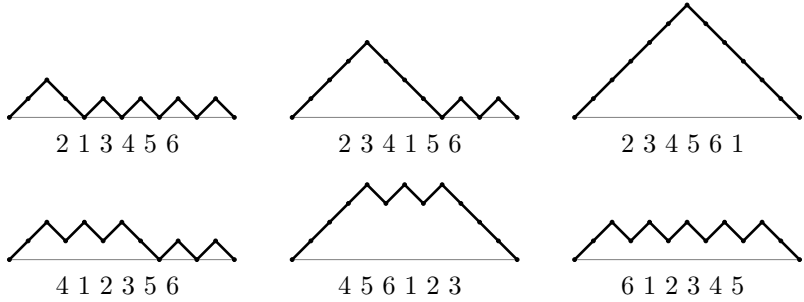

\small
\begin{tabular}{ccc}
\dyckpath{0,0}{6}{1,1,0,0,1,0,1,0,1,0,1,0}{\eps} & \dyckpath{0,0}{6}{1,1,1,1,0,0,0,0,1,0,1,0}{\eps} & \dyckpath{0,0}{6}{1,1,1,1,1,1,0,0,0,0,0,0}{\eps} \\
2 1 3 4 5 6 & 2 3 4 1 5 6 & 2 3 4 5 6 1 \\[8pt]
\dyckpath{0,0}{6}{1,1,0,1,0,1,0,0,1,0,1,0}{\eps} & \dyckpath{0,0}{6}{1,1,1,1,0,1,0,1,0,0,0,0}{\eps} & \dyckpath{0,0}{6}{1,1,0,1,0,1,0,1,0,1,0,0}{\eps} \\
  4 1 2 3 5 6 & 4 5 6 1 2 3 & 6 1 2 3 4 5
\end{tabular}
\caption{Elements of $\Gr^{odd}_6(132)$ and their Dyck paths.}
\end{figure}

\begin{prop} \label{prop:213dyckpaths}
The elements of $\Gr^{odd}_{n}(213)$ are in one-to-one correspondence with the odd Grassmannian Dyck paths of semilength $n$ that have exactly one long ascent and one long descent, with no peak after the long descent.
\end{prop}
\begin{figure}[ht]
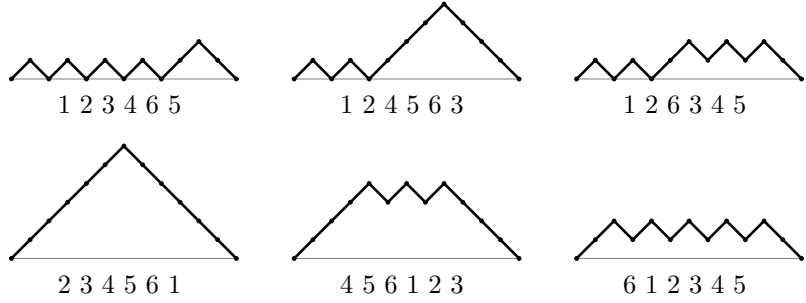

\small
\begin{tabular}{ccc}
\dyckpath{0,0}{6}{1,0,1,0,1,0,1,0,1,1,0,0}{\eps} & \dyckpath{0,0}{6}{1,0,1,0,1,1,1,1,0,0,0,0}{\eps} & \dyckpath{0,0}{6}{1,0,1,0,1,1,0,1,0,1,0,0}{\eps} \\
1 2 3 4 6 5 & 1 2 4 5 6 3 & 1 2 6 3 4 5 \\[8pt]
\dyckpath{0,0}{6}{1,1,1,1,1,1,0,0,0,0,0,0}{\eps} & \dyckpath{0,0}{6}{1,1,1,1,0,1,0,1,0,0,0,0}{\eps} & \dyckpath{0,0}{6}{1,1,0,1,0,1,0,1,0,1,0,0}{\eps} \\
  2 3 4 5 6 1 & 4 5 6 1 2 3 & 6 1 2 3 4 5
\end{tabular}
\caption{Elements of $\Gr^{odd}_6(213)$ and their Dyck paths.}
\end{figure}

\begin{thm}\label{thm:312avoiders} %ref OEIS A002620
For $n\ge 2$ we have
\begin{equation*}
  \abs{\Gr^{odd}_{n}(312)} = \binom{\lfloor\frac{n-1}{2}\rfloor+1}{2} + \binom{\lfloor\frac{n}{2}\rfloor+1}{2}.
\end{equation*}
\end{thm}
\begin{proof}
Excluding the identity, any Grassmannian permutation that avoids $312$ can be written as
\begin{equation*}\label{312shape}
 \underbrace{1 \cdots (i-1)}_{\tau_{1}} \, \underbrace{(i+1) \cdots m}_{\tau_{2}} \, i \, \underbrace{(m+1) \cdots n}_{\tau_{3}} \;\text{ with } i,m\in\{1,\dots,n\},\; i<m,
\end{equation*}
where $\tau_{1} = \varepsilon$ if $i=1$, and $\tau_{3} = \varepsilon$ if $m = n$. Moreover, the permutation is odd if and only if $|\tau_2|$ is an odd number.

Since $\abs{\Gr_{n}(312)} = \binom{n}{2}+1$,
\begin{equation*} \label{312begeqn}
    \binom{n}{2}+1 = \abs{\Gr_{n}^{odd}(312)} +\abs{\Gr_{n}^{even,*}(312)}+1.
\end{equation*}
where $\Gr_{n}^{even,*}(312) = \Gr_{n}^{even}(312)\backslash\{\id_n\}$. Define $\alpha:\Gr_{n}^{even,*}(312) \to \Gr_{n-1}^{odd}(312)$ as follows. 

Let $\tau\in \Gr_{n}^{even,*}(312)$, so $|\tau_2|\ge 2$. If $m=n$ (i.e.\ $\tau_3=\varepsilon$), we let $\alpha(\tau)$ be the permutation in $\Gr_{n-1}(312)$ obtained by removing $n$ from $\tau$. This reduces $\tau_2$ by one, so $\alpha(\tau)$ is odd.

If $m<n$ ($\tau_3\not=\varepsilon$), then $\tau$ has $n$ at the end. In this case, we let $\alpha(\tau)$ be the permutation obtained by removing $n$ from $\tau$ and moving the entry $m$ to the right of $i$. Again, $\alpha(\tau)$ is in $\Gr_{n-1}(312)$ and has one less inversion than $\tau$, so $\alpha(\tau)$ is odd. Note that $\tau$ ends with a descent ($\tau_3=\varepsilon$) if and only if $\alpha(\tau)$ does. In other words, the map $\alpha$ is reversible.

Now, since $\Gr_{n-1}^{odd}(312) \cong \Gr_{n}^{even,*}(312)$, we conclude
\begin{equation*}
\abs{\Gr^{odd}_{n}(312)} = \binom{n}{2} - \abs{\Gr^{odd}_{n-1}(312)}.
\end{equation*}
This relation leads to the generating function given in Table~\ref{tab:size3patt}, which is $(x+1)$ times the generating function for $\abs{\Gr_{n}^{odd}(132)}$. So, the stated formula follows from Theorem~\ref{thm:132avoiders}.
\end{proof}

The following two propositions can be easily verified.
\begin{prop} \label{prop:312dyckpaths}
The elements of $\Gr^{odd}_{n}(312)$ are in one-to-one correspondence with the odd Grassmannian Dyck paths of semilength $n$ that have no valleys above height zero. 
\end{prop}
\begin{figure}[ht]
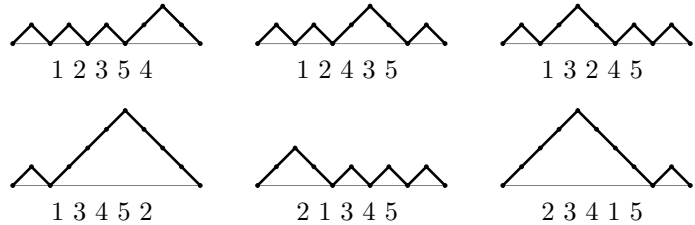

\small
\begin{tabular}{ccc}
\dyckpath{0,0}{5}{1,0,1,0,1,0,1,1,0,0}{\eps} & \dyckpath{0,0}{5}{1,0,1,0,1,1,0,0,1,0}{\eps} & \dyckpath{0,0}{5}{1,0,1,1,0,0,1,0,1,0}{\eps} \\
 1 2 3 5 4 & 1 2 4 3 5 & 1 3 2 4 5 \\[8pt]
 \dyckpath{0,0}{5}{1,0,1,1,1,1,0,0,0,0}{\eps} & \dyckpath{0,0}{5}{1,1,0,0,1,0,1,0,1,0}{\eps} & \dyckpath{0,0}{5}{1,1,1,1,0,0,0,0,1,0}{\eps} \\
 1 3 4 5 2 & 2 1 3 4 5 & 2 3 4 1 5 
\end{tabular}
\caption{Elements of $\Gr^{odd}_5(312)$ and their Dyck paths.}
\end{figure}

\begin{prop} \label{prop:231dyckpaths}
The elements of $\Gr^{odd}_{n}(231)$ are in one-to-one correspondence with the odd Grassmannian Dyck paths of semilength $n$ that have no peaks above height two. 
\end{prop}
\begin{figure}[ht]
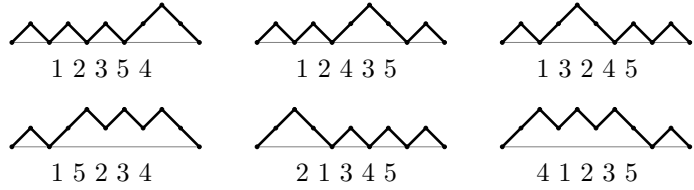

\small
\begin{tabular}{ccc}
\dyckpath{0,0}{5}{1,0,1,0,1,0,1,1,0,0}{\eps} & \dyckpath{0,0}{5}{1,0,1,0,1,1,0,0,1,0}{\eps} & \dyckpath{0,0}{5}{1,0,1,1,0,0,1,0,1,0}{\eps}\\
 1 2 3 5 4 & 1 2 4 3 5 & 1 3 2 4 5  \\[8pt]
 \dyckpath{0,0}{5}{1,0,1,1,0,1,0,1,0,0}{\eps} & \dyckpath{0,0}{5}{1,1,0,0,1,0,1,0,1,0}{\eps} & \dyckpath{0,0}{5}{1,1,0,1,0,1,0,0,1,0}{\eps}\\
 1 5 2 3 4 & 2 1 3 4 5 & 4 1 2 3 5\\
\end{tabular}
\caption{Elements of $\Gr^{odd}_5(231)$ and their Dyck paths}
\end{figure}

%%%%%%%%%%%%%%%%%%%%%%%%%%%%%%%%%%%%%%%%%%%%
\section{Patterns of size 4}
\label{sec:size4patts}

Once again, recall that $\abs{\Gr^{odd}_{1}(\sigma)}=0$ and $\abs{\Gr^{odd}_{2}(\sigma)}=1$ for any $\sigma$ with $\abs{\sigma}\ge 3$.

\begin{thm} %ref OEIS A175287
For $\sigma\in\{1243,2134,2341,4123\}$ and $n\ge 2$, we have
\[ \abs{\Gr^{odd}_{n}(\sigma)} = \binom{n}{3} + n - 1 - \abs{\Gr^{odd}_{n-1}(\sigma)}. \]
\end{thm}
\begin{proof}
We will focus on the patterns $2341$ and $1243$. The reverse complement map gives the statement for the other two patterns. 

Observe that every $\tau\in \Gr_{n}(2341)$, excluding the identity, must be of the form 
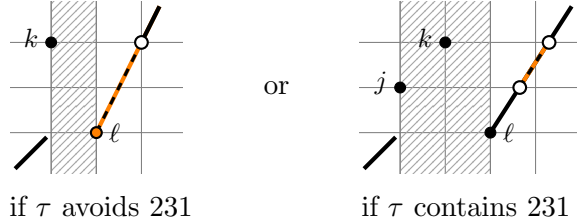
\begin{center}
\begin{tikzpicture}[scale=0.6]
\begin{scope}
\clip (0.1,0.1) rectangle (3.9,3.9);
\draw[gray] (0,0) grid (4,4);
\draw[mesh] (1,0) rectangle (2,4);
\draw[ultra thick] (0.2,0.2) -- (0.9,0.9);
\draw[ultra thick, orange] (2,1) -- (3.4,3.8);
\draw[very thick, dashed] (2,1) -- (3.4,3.8);
\draw[ultra thick] (3,3) -- (3.4,3.8);
\draw[fill] (1,3) circle (0.12);
\draw[fill,orange] (2,1) circle (0.12);
\draw[thick] (2,1) circle (0.12);
\draw[fill=white,thick] (3,3) circle (0.14);
\node[left=1pt] at (1,3.1) {\small $k$};
\node[right=1pt] at (2,1) {\small $\ell$};
\end{scope}

\node at (6,2) {or};
\node[below=3pt] at (60pt,0) {if $\tau$ avoids $231$};
\node[below=3pt] at (290pt,0) {if $\tau$ contains $231$};

\begin{scope}[xshift=220]
\clip (0.1,0.1) rectangle (4.9,3.9);
\draw[gray] (0,0) grid (5,4);
\draw[mesh] (1,0) rectangle (3,4);
\draw[ultra thick] (0.2,0.2) -- (0.9,0.9);
\draw[ultra thick] (3,1) -- (4.8,3.8);
\draw[line width=1.62pt, orange] (3.643,2) -- (4.286,3);
\draw[line width=1.3pt, dashed] (3.643,2) -- (4.286,3);
\foreach \x/\y in {1/2,2/3,3/1}{\draw[fill] (\x,\y) circle (0.12);}
\draw[fill=white,thick] (3.65,2) circle (0.14);
\draw[fill=white,thick] (4.3,3) circle (0.14);
\node[left=1pt] at (1,2.1) {\small $j$};
\node[left=1pt] at (2,3.1) {\small $k$};
\node[right=1pt] at (3,1) {\small $\ell$};
\end{scope}
\end{tikzpicture}
\end{center}
where a line segment (solid or dashed) represents an increasing sequence of consecutive numbers, and the length of the dashed portion (black \& orange) determines the parity of the permutation. For instance, such a permutation is odd if and only if the dashed segment has odd length. Note that every permutation in $\Gr^{odd}_{n-1}(2341)$ gives rise to a unique permutation in $\Gr^{even}_{n}(2341)$, created by adding one element to its dashed segment. Moreover, the only even permutations of size $n$ that cannot be generated through this process are those for which the dashed segment is actually empty. Let $\A_n$ denote the set of such permutations. We have $\id_n\in \A_n$, and if $\tau\in \A_n$ is not the identity, then it must be of the form 
\[ \tau =  \tau_1\, j\,(j+1)\,\ell\, \tau_2\tau_3 \;\text{ with $j,\ell\in\{1,2,\dots,n-1\}$ and $j>\ell$}, \]
where every $\tau_i$ is either $\varepsilon$ or an increasing run of consecutive numbers. There are $\binom{n-1}{2}$ such permutations, hence $\abs{\Gr^{even}_{n}(2341)} = \abs{\Gr^{odd}_{n-1}(2341)} + \binom{n - 1}{2} + 1$.
Therefore, 
\begin{align*}
  \abs{\Gr^{odd}_{n}(2341)} &= \abs{\Gr_{n}(2341)} - \abs{\Gr^{even}_{n}(2341)} \\
  &= 1 + \binom{n}{2} + \binom{n}{3} - \abs{\Gr^{odd}_{n-1}(2341)} - \binom{n - 1}{2} - 1 \\
  &= \binom{n}{3} + n -1 - \abs{\Gr^{odd}_{n-1}(2341)}.
\end{align*}

On the other hand, every permutation in $\Gr_{n}(1243)$ must be of one of the following forms:

\smallskip
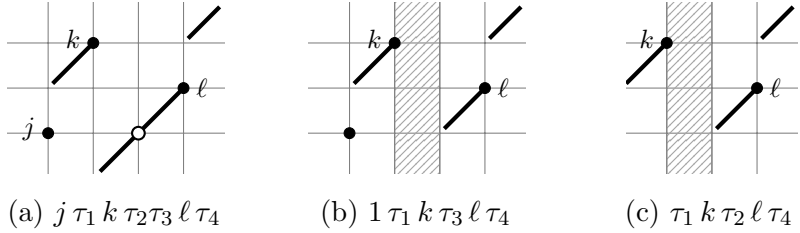
\begin{center}
\begin{tikzpicture}[scale=0.6]
\begin{scope}
\clip (0.1,0.1) rectangle (4.9,3.9);
\draw[gray] (0,0) grid (5,4);
\draw[ultra thick] (1.1,2.1) -- (2,3);
\draw[ultra thick] (2.15,0.15) -- (4,2);
\draw[ultra thick] (4.1,3.1) -- (4.8,3.8);
\foreach \x/\y in {1/1,2/3,4/2}{\draw[fill] (\x,\y) circle (0.12);}
\draw[fill=white,thick] (3,1) circle (0.14);
\node[left=1pt] at (1,1.1) {\small $j$};
\node[left=1pt] at (2,3.1) {\small $k$};
\node[right=1pt] at (4,2) {\small $\ell$};
\end{scope}
\node[below=4pt] at (70pt,0) {(a) $j\,\tau_1\,k\,\tau_2 \tau_3\,\ell \,\tau_4$};
\node[below=4pt] at (260pt,0) {(b) $1\,\tau_1\,k\, \tau_3\,\ell \,\tau_4$};
\node[below=4pt] at (445pt,0) {(c) $\tau_1\,k\, \tau_2\,\ell \,\tau_4$};
\begin{scope}[xshift=190]
\clip (0.1,0.1) rectangle (4.9,3.9);
\draw[gray] (0,0) grid (5,4);
\draw[mesh] (2,0) rectangle (3,4);
\draw[ultra thick] (1.1,2.1) -- (2,3);
\draw[ultra thick] (3.1,1.1) -- (4,2);
\draw[ultra thick] (4.1,3.1) -- (4.8,3.8);
\foreach \x/\y in {1/1,2/3,4/2}{\draw[fill] (\x,\y) circle (0.12);}
\node[left=1pt] at (2,3.1) {\small $k$};
\node[right=1pt] at (4,2) {\small $\ell$};
\end{scope}
\begin{scope}[xshift=390]
\clip (0.1,0.1) rectangle (3.9,3.9);
\draw[gray] (0,0) grid (4,4);
\draw[mesh] (1,0) rectangle (2,4);
\draw[ultra thick] (0.1,2.1) -- (1,3);
\draw[ultra thick] (2.1,1.1) -- (3,2);
\draw[ultra thick] (3.1,3.1) -- (3.8,3.8);
\foreach \x/\y in {1/3,3/2}{\draw[fill] (\x,\y) circle (0.12);}
\node[left=1pt] at (1,3.1) {\small $k$};
\node[right=1pt] at (3,2) {\small $\ell$};
\end{scope}
\end{tikzpicture}
\end{center}
Every $\tau\in \Gr^{even,*}_{n-1}(1243)$ gives rise to a unique permutation $\gamma(\tau)$ in $\Gr^{odd}_{n}(1243)$, constructed as follows:
\begin{itemize}
\item[$\triangleright$] If $|\tau_1|$ is even, increase each term greater than $\ell$ by one and insert $\ell+1$ at the position to the right of $\ell$. For example, $\gamma(2567134)=2678134{\color{red}5}$, $\gamma(1456237)=156723{\color{red}4}8$, and $\gamma(5671234)=6781234{\color{red}5}$.

\item[$\triangleright$] If $\tau$ is of type (a) or (b) and $|\tau_1|$ is odd, increase each element by one and insert 1 at the position of the descent of $\tau$. For instance, we have $\gamma(245136)=356{\color{red}1}247$ and $\gamma(145236)=256{\color{red}1}347$.

\item[$\triangleright$] If $\tau$ is of type (c) and $|\tau_1|$ is odd, increase each element greater than $\ell+1$ by one and insert $\ell+2$ at the position to the right of $\ell$. For example, $\gamma(345612)=356712{\color{red}4}$.
\end{itemize}
In all three cases, the above algorithm creates an odd number of new inversions, so $\tau$ and $\gamma(\tau)$ have different parity.

Let $\B_n$ be the set of elements in $\Gr^{odd}_{n}(1243)$ of the form $1\,\tau_1\,k\, 2\,\tau_4$ or $\tau_1\,k\, 1\,\tau_4$. These are special cases of permutations of type (b) and (c), and they cannot be created from an element in $\Gr^{even,*}_{n-1}(1243)$ through the above algorithm. Every odd $1<k\le n$ produces an odd permutation of the form $1\,\tau_1\,k\, 2\,\tau_4$, so there are $\lfloor \frac{n-1}{2}\rfloor$ of them. On the other hand, every even $k\le n$ gives an odd permutation of the form $\tau_1\,k\, 1\,\tau_4$, so there are $\lfloor \frac{n}{2}\rfloor$ of those. In conclusion, $\abs{\B_n}= \lfloor \frac{n-1}{2}\rfloor+\lfloor \frac{n}{2}\rfloor = n-1$.

Finally, it is easy to see that $\gamma: \Gr^{even,*}_{n-1}(1243)\to \Gr^{odd}_{n}(1243)\backslash \B_n$ is bijective, hence
\[  \abs{\Gr^{odd}_{n}(1243)} = \abs{\Gr^{even,*}_{n-1}(1243)} + n - 1. \]
The claimed formula follows from the fact that $\abs{\Gr^{even,*}_{n-1}(\sigma)} = \binom{n}{3} - \abs{\Gr^{odd}_{n-1}(\sigma)}$.
\end{proof}

\medskip
Before discussing the next set of patterns, recall that if $|\sigma|=4$ and $\des(\sigma)=1$, then
$\abs{\Gr^*_{n}(\sigma)} = \binom{n+1}{3}$. In other words, for $n\ge 2$, the set $\Gr^*_{n}(\sigma)$ is equinumerous with the set $\W_{4}(n-2)$ of weak compositions of $n-2$ having 4 parts. For example, every permutation $\tau\in \Gr^*_{n}(3412)$ must be of the form 

\medskip
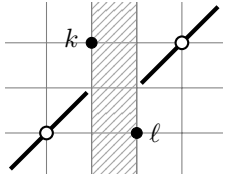
\begin{center}
\begin{tikzpicture}[scale=0.6]
\clip (0.1,0.1) rectangle (4.9,3.9);
\draw[gray] (0,0) grid (5,4);
\draw[mesh] (2,0) rectangle (3,4);
\draw[ultra thick] (0.2,0.2) -- (1.9,1.9);
\draw[ultra thick] (3.1,2.1) -- (4.8,3.8);
\foreach \x/\y in {2/3,3/1}{\draw[fill] (\x,\y) circle (0.12);}
\draw[fill=white,thick] (1,1) circle (0.14);
\draw[fill=white,thick] (4,3) circle (0.14);
\node[left=1pt] at (2,3.1) {\small $k$};
\node[right=1pt] at (3,1) {\small $\ell$};
\end{tikzpicture}
\end{center}
and can be written as
\[ \tau = \tau_1 \tau_2\, k\, \ell\, \tau_3 \tau_4 \;\text{ with } k>\ell. \]
Similarly, every $\tau\in \Gr^*_{n}(1423)$ is of the form

\medskip
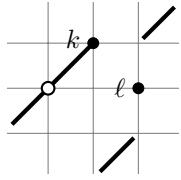
\begin{center}
\begin{tikzpicture}[scale=0.6]
\clip (0.1,0.1) rectangle (3.9,3.9);
\draw[gray] (0,0) grid (4,4);
\draw[ultra thick] (0.2,1.2) -- (2,3);
\draw[ultra thick] (2.15,0.15) -- (2.9,0.9);
\draw[ultra thick] (3.1,3.1) -- (3.8,3.8);
\foreach \x/\y in {2/3,3/2}{\draw[fill] (\x,\y) circle (0.12);}
\draw[fill=white,thick] (1,2) circle (0.14);
\node[left=1pt] at (2,3.1) {\small $k$};
\node[left=1pt] at (3,2) {\small $\ell$};
\end{tikzpicture}
\end{center}
and can be written as
\[ \tau = \tau_1 \tau_2\, k\, \tau_3\, \ell\, \tau_4 \;\text{ with } k>\ell. \]
In both cases, the map
\begin{equation*}
\Lambda^{\sigma}(\tau) = (|\tau_1|,|\tau_2|,|\tau_3|,|\tau_4|)
\end{equation*}
gives a bijection $\Gr^*_{n}(\sigma)\to \W_{4}(n-2)$. Moreover, 
\begin{align} \notag
 \tau\in\Gr_{n}(3412) \text{ is odd} &\iff |\tau_2| \equiv |\tau_3| \!\pmod{2}, \\ \label{eq:1423odd}
 \tau\in\Gr_{n}(1423) \text{ is odd} &\iff (|\tau_1|+|\tau_2|+1)|\tau_3| + |\tau_2|\equiv 0 \!\pmod{2}.
\end{align}

\begin{thm}\label{thm:3412oddGrass} %ref OEIS A005993
For $\sigma\in\{1423, 2314, 3412\}$ and $n\ge 1$, we have
\[ \abs{\Gr^{odd}_{2n-1}(\sigma)} = \frac12 \binom{2n}{3} \;\text{ and }\; \abs{\Gr^{odd}_{2n}(\sigma)} =
\frac12 \left[\binom{2n+1}{3} + n\right]. \]
\end{thm}
\begin{proof}
For simplicity of notation, we write $\Lambda=\Lambda^{\sigma}$.  For $\hat t =  (t_1,t_2,t_3,t_4)$ with $t_i\ge 0$, let 
\begin{equation*}
\phi(\hat t) = 
\begin{cases}
 (t_1,t_2,t_3+1,t_4-1)& \text{if $t_3$ even and $t_4>0$,} \\
 (t_1-1,t_2+1,t_3,t_4)& \text{if $t_3$ even and $t_4=0$,} \\
 (t_1,t_2,t_3-1,t_4+1)& \text{if $t_3$ odd}.
\end{cases}
\end{equation*}
If $\hat t = \Lambda(\tau)$ with $\tau\in \Gr^{odd}_{2n-1}(3412)$, then $t_2\equiv t_3\pmod{2}$, and so $\phi(\hat t)_2\not\equiv \phi(\hat t)_3\pmod{2}$. Moreover, in this case, $\sum t_i$ is odd, so if $t_2$ and $t_3$ are both even and $t_4=0$, then $t_1$ must be odd and thus greater than 0. In conclusion, $\Lambda^{-1}(\phi(\hat t))$ is an element of $\Gr^{even,*}_{2n-1}(3412)$. Clearly, the map $\phi: \Lambda(\Gr^{odd}_{2n-1}(3412))\to \Lambda(\Gr^{even,*}_{2n-1}(3412))$ is invertible and 
\[ \Lambda^{-1}\circ\phi\circ\Lambda: \Gr^{odd}_{2n-1}(3412)\to \Gr^{even,*}_{2n-1}(3412) \] 
is a bijection. Let $\A_{2n}$ be the set of permutations $\tau\in\Gr^{odd}_{2n}(3412)$ with $\tau_1=\tau_4=\varepsilon$ and such that $|\tau_2|\equiv|\tau_3|\equiv 0\!\pmod{2}$. For example, $\A_8= \{81234567, 23814567, 23458167, 23456781\}$. We have $\abs{\A_{2n}}=n$, and if $\tau\not\in\A_{2n}$, then $\Lambda(\tau)$ cannot be of the form $(0,t_2,t_3,0)$ with $t_3$ even. The above map $\phi$ gives a bijection $\Lambda^{-1}\circ\phi\circ\Lambda: \Gr^{odd}_{2n}(3412)\backslash \A_{2n} \to \Gr^{even,*}_{2n}(3412)$.

In conclusion,
\[ \abs{\Gr^{odd}_{2n-1}(3412)} =  \abs{\Gr^{even}_{2n-1}(3412)} - 1, \quad \abs{\Gr^{odd}_{2n}(3412)} =  \abs{\Gr^{even}_{2n}(3412)} + n-1, \]
and the claimed formulas for $\sigma=3412$ both follow from the identities $\abs{\Gr_m(\sigma)} = 1+\binom{m+1}{3}$ and $\abs{\Gr_m(\sigma)} = \abs{\Gr^{even}_m(\sigma)} + \abs{\Gr^{odd}_m(\sigma)}$.

The case when $\sigma=1423$ can be treated similarly. For $\hat t =  (t_1,t_2,t_3,t_4)$, we now define
\begin{equation*}
\psi(\hat t) = 
\begin{cases}
 (t_1-1,t_2+1,t_3,t_4)& \text{if $t_1> 0$,} \\
 (t_2,0,t_3+1,t_4-1)& \text{if $t_1=0$}.
\end{cases}
\end{equation*}
If $\hat t = \Lambda(\tau)$ with $\tau\in \Gr^{odd}_{2n-1}(1423)$, \eqref{eq:1423odd} implies $(t_1+t_2+1)t_3+t_2\equiv 0\!\pmod{2}$. If $t_1>0$, then $(\psi(\hat t)_1+\psi(\hat t)_2+1)\psi(\hat t)_3+\psi(\hat t)_2\equiv 1\!\pmod{2}$. If $t_1=0$, then $t_2$ and $t_3$ must both be even, so $\psi(\hat t)_3$ is odd and again $(\psi(\hat t)_1+\psi(\hat t)_2+1)\psi(\hat t)_3+\psi(\hat t)_2\equiv 1\!\pmod{2}$. Moreover, since $\sum t_i$ is odd, $t_4$ must be odd and thus greater than 0. In conclusion, 
\[ \Lambda^{-1}\circ\psi\circ\Lambda: \Gr^{odd}_{2n-1}(1423)\to \Gr^{even,*}_{2n-1}(1423) \] 
is a bijective map. Now, let $\B_{2n}$ consist of all $\tau\in\Gr^{odd}_{2n}(1423)$ with $\tau_1=\tau_4=\varepsilon$. For example, $\B_8= \{81234567, 67812345, 45678123, 23456781\}$. Clearly, $\abs{\B_{2n}}=n$, and $\psi$ gives a bijection $\Lambda^{-1}\circ\psi\circ\Lambda: \Gr^{odd}_{2n}(1423)\backslash \B_{2n} \to \Gr^{even,*}_{2n}(1423)$. Note that if $\tau\not\in \B_{2n}$ and $t_1=0$, then $t_4>0$. From here, we can argue as for $\sigma=3412$. 

Finally, the statement for $\sigma=2314$ follows using the reverse complement map.
\end{proof}

\begin{thm} %ref OEIS A006918
For $n\ge 2$ we have
\[ \abs{\Gr^{odd}_{n}(2413)} =  \abs{\Gr^{odd}_{n-1}(2413)} + \binom{\lfloor\frac{n}{2}\rfloor+1}{2}. \]
\end{thm}
\begin{proof}
Let $\E_n$ be the set of permutations in $\Gr^{odd}_{n}(2413)$ that contain a $132$ pattern. Then, since $S_{n}(132)\subset S_{n}(2413)$, we have $\abs{\Gr^{odd}_{n}(2413)} =  \abs{\E_n} + \abs{\Gr^{odd}_{n}(132)}$. 

Clearly, for every element $\tau\in \Gr^{odd}_{n-1}(2413)$, we have $1\oplus\tau\in \E_n$. On the other hand, every permutation in $\E_n$ must start with 1 (otherwise, 1 would be part of the descent creating a $2413$ pattern). In other words, the map $\tau\mapsto 1\oplus\tau$ gives a bijection from $\Gr^{odd}_{n-1}(2413)$ to $\E_n$, and the statement of the theorem follows by means of Theorem~\ref{thm:132avoiders}.
\end{proof}

\begin{thm}
For $n\ge 1$ we have
\[ \abs{\Gr^{odd}_{2n-1}(1324)} = \abs{\Gr^{odd}_{2n-1}(2413)}  \;\text{ and }\;  \abs{\Gr^{odd}_{2n}(1324)} = \abs{\Gr^{odd}_{2n}(2134)}. \]
\end{thm}
\begin{proof}
To prove both identities we will provide explicit bijections between the sets involved. We start by pointing out that the plot of any permutation in $\Gr^{odd}_{n}(1324)$ looks like
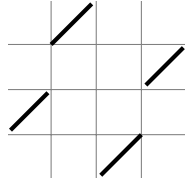
\begin{center}
\begin{tikzpicture}[scale=0.6]
\clip (0.05,0.05) rectangle (3.95,3.95);
\draw[gray] (0,0) grid (4,4);
\draw[ultra thick] (0.1,1.1) -- (0.93,1.93);
\draw[ultra thick] (1,3) -- (1.9,3.9);
\draw[ultra thick] (2.1,0.1) -- (3,1);
\draw[ultra thick] (3.1,2.1) -- (3.93,2.93);
\end{tikzpicture}
\end{center}
and can therefore be written as an inflation of $2413$ by empty or identity permutations.
Let $\varphi_1:\Gr^{odd}_{2n-1}(2413)\to \Gr^{odd}_{2n-1}(1324)$ be defined as follows. If $\tau\in \Gr^{odd}_{2n-1}(2413)$ avoids the pattern $1324$, we let $\varphi_1(\tau)=\tau$. If $\tau\in \Gr^{odd}_{2n-1}(2413)$ contains the pattern $1324$, it must be of the form
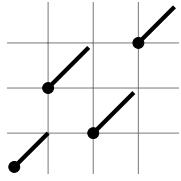
\begin{center}
\begin{tikzpicture}[scale=0.6]
\clip (0.1,0.1) rectangle (3.9,3.9);
\draw[gray] (0,0) grid (4,4);
\draw[ultra thick] (0.2,0.2) -- (1,1);
\draw[ultra thick] (1,2) -- (1.9,2.9);
\draw[ultra thick] (2,1) -- (2.9,1.9);
\draw[ultra thick] (3,3) -- (3.8,3.8);
\foreach \x/\y in {0.25/0.25,1/2,2/1,3/3}{\draw[fill] (\x,\y) circle (0.12);}
\end{tikzpicture}
\end{center}
and can then be written as 
\[ \tau = 1324[\id_{k_1},\id_{k_2},\id_{k_3},\id_{k_4}], \]
where $k_j\ge 1$ and $\sum k_j=2n-1$. In this case, we let $\varphi_1(\tau)=2413[\id_{k_2},\id_{k_1},\id_{k_4},\id_{k_3}]$. Note that $\varphi_1(\tau)$ has $(k_2+k_1)k_4+k_1k_3$ inversions. Moreover, $\tau$ odd implies $k_2\equiv k_3\equiv 1$, and $\sum k_j\equiv 1$  implies $k_1\not\equiv k_4\pmod{2}$. Thus $k_1+k_4\equiv 1$, $k_1k_4\equiv 0$, and therefore,
\[ \inv(\varphi_1(\tau)) = (k_2+k_1)k_4+k_1k_3 \equiv (1+k_1)k_4+k_1 \equiv 1\!\pmod{2}. \]
Hence $\varphi_1(\tau)\in \Gr^{odd}_{2n-1}(1324)$. In the definition of $\varphi_1$, the first case produces permutations that avoid $2413$ while the second case gives permutations that contain a $2413$ pattern. Therefore, the map $\varphi_1:\Gr^{odd}_{2n-1}(2413)\to \Gr^{odd}_{2n-1}(1324)$ is bijective.

We now proceed to define a map $\varphi_2:\Gr^{odd}_{2n}(2134)\to \Gr^{odd}_{2n}(1324)$. As before, we let $\varphi_2(\tau)=\tau$ whenever $\tau$ avoids both patterns. If $\tau\in \Gr^{odd}_{2n}(2134)$ contains the pattern $1324$, it must be of the form
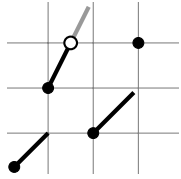
\begin{center}
\begin{tikzpicture}[scale=0.6]
\clip (0.1,0.1) rectangle (3.9,3.9);
\draw[gray] (0,0) grid (4,4);
\draw[ultra thick] (0.2,0.2) -- (1,1);
\draw[ultra thick] (1,2) -- (1.5,3);
\draw[ultra thick, gray!80] (1.5,3) -- (1.9,3.8);
\draw[ultra thick] (2,1) -- (2.9,1.9);
\foreach \x/\y in {0.25/0.25,1/2,2/1,3/3}{\draw[fill] (\x,\y) circle (0.12);}
\draw[fill=white,thick] (1.5,3) circle (0.14);
\end{tikzpicture}
\end{center}
where the gray segment could be empty. Thus, it can be written as
\[ \tau = 13524[\id_{k_1},\id_{k_2},\id_{k_3},\id_{k_4},1], \]
where $k_1,k_2,k_4\ge 1$, $k_3\ge 0$, $\sum k_j = 2n-1$, and $(k_2+k_3)k_4 + k_3\equiv 1\pmod{2}$. Here we adopt the convention that $\id_0=\varepsilon$. In this case, we define
\begin{equation*}
 \varphi_2(\tau) = 
 \begin{cases}
  2413[\id_{k_2},\id_{k_3},\id_{k_1},\id_{k_4+1}], & \text{if } k_3\not\equiv k_4\!\!\pmod{2}, \\
  2413[\id_{k_2+1},\id_{k_3},\id_{k_1-1},\id_{k_4+1}], & \text{if } k_3\equiv k_4\equiv 1\!\!\pmod{2}.
 \end{cases}
\end{equation*}
If $k_3\not\equiv k_4$, then $k_3+k_4\equiv 1$ and $k_3k_4\equiv 0\pmod{2}$. $\sum k_j = 2n-1$ then implies $k_1\equiv k_2$, and since $(k_2+k_3)k_4 + k_3\equiv 1$, we must have $k_3\equiv 1+k_1k_4\pmod{2}$. Therefore,
\begin{align*}
 \inv(\varphi_2(\tau)) &= (k_2+k_3)k_1+k_3(k_4+1) \\
 &\equiv  (k_1+k_3+k_4)k_1+1 \!\pmod{2}\\
 &\equiv  (k_1+1)k_1+1 \equiv 1 \!\pmod{2}.
\end{align*}
On the other hand, if $k_3\equiv k_4\equiv 1$, then $k_2\equiv 1$, $k_1\equiv 0$ (in particular, $k_1\ge 2$), and so
\[ \inv(\varphi_2(\tau)) = (k_2+1+k_3)(k_1-1)+k_3(k_4+1) \equiv 1\!\pmod{2}. \]
By construction, $\varphi_2(\tau)$ is odd, avoids $1324$, and it contains a $2134$ pattern if and only if $\tau$ contains a $1324$ pattern. The map $\varphi_2:\Gr^{odd}_{2n}(2134)\to \Gr^{odd}_{2n}(1324)$ is a bijection.
\end{proof}

\begin{thm}
For $\sigma\in\{1342, 3124\}$ and $n\ge 1$, we have
\[ \abs{\Gr^{odd}_{2n+1}(\sigma)} = \abs{\Gr^{odd}_{2n+1}(2341)}  \;\text{ and }\;  \abs{\Gr^{odd}_{2n}(\sigma)} = \abs{\Gr^{odd}_{2n}(1423)}. \]
\end{thm}
\begin{proof}
First, the reverse complement map gives the identity $\abs{\Gr^{odd}_{n}(3124)}=\abs{\Gr^{odd}_{n}(1342)}$.

Similar to the proof of the previous theorem, we will provide explicit bijections 
\[ \psi_1:\Gr^{odd}_{2n+1}(2341)\to \Gr^{odd}_{2n+1}(1342) \text{ and } 
\psi_2:\Gr^{odd}_{2n}(1423)\to \Gr^{odd}_{2n}(1342). \]
It is easy to verify that every element of $\Gr^{odd}_{n}(1342)$ can be written as an inflation of $24135$ by empty or identity permutations.

Let $\tau\in \Gr^{odd}_{2n+1}(2341)$. If $\tau$ avoids $1342$, we let $\psi_1(\tau)=\tau$. If $\tau$ contains $1342$, then it must be of the form $\tau = 135246[\id_{k_1},1,1,\id_{k_2},\id_{k_3},\id_{k_4}]$ with $k_1,k_2,k_3\ge 1$, $k_4\ge 0$, $\sum k_j=2n-1$, and $k_3$ odd. In this case, we let
\begin{equation*}
 \psi_1(\tau) = 
 \begin{cases}
  24135[\id_{k_2+1},1,\id_{k_1},\id_{k_3},\id_{k_4}], & \text{if } k_1k_2\equiv 0\!\!\pmod{2}, \\
  24135[\id_{k_2+1},1,\id_{k_1},\id_{k_3-1},\id_{k_4+1}], & \text{if } k_1 k_2\equiv 1\!\!\pmod{2}.
 \end{cases}
\end{equation*}
Moreover,
\begin{equation*}
\inv(\psi_1(\tau)) =
 \begin{cases}
   (k_2+2)k_1 + k_3\equiv 1, & \text{if } k_1k_2\equiv 0\!\!\pmod{2}, \\
   (k_2+2)k_1 + k_3-1\equiv 1, & \text{if } k_1 k_2\equiv 1\!\!\pmod{2}.
 \end{cases}
\end{equation*}
In all cases, $\psi_1(\tau)$ is odd, avoids $1342$, and it contains a $2341$ pattern if and only if $\tau$ contains a $1342$ pattern. It is not difficult to see that $\psi_1$ is a bijective map.

Now, let $\tau\in \Gr^{odd}_{2n}(1423)$. If $\tau$ avoids $1342$, we let $\psi_2(\tau)=\tau$. If $\tau$ contains $1342$, then it must be of the form $\tau = 24135[\id_{k_1}, \id_{k_2},\id_{k_3},1,\id_{k_4}]$ with $k_1\ge 1$, $k_2\ge 2$, $k_3,k_4\ge 0$, $\sum k_j=2n-1$, and $(k_1+k_2)k_3 + k_2\equiv 1\pmod{2}$. In this case, we let
\begin{equation*}
 \psi_2(\tau) = 
 \begin{cases}
  24135[\id_{k_1},1,\id_{k_3},\id_{k_2},\id_{k_4}], & \text{if } k_2\equiv 1\!\!\pmod{2}, \\
  24135[\id_{k_1+1},1,\id_{k_3},\id_{k_2},\id_{k_4-1}], & \text{if } k_2\equiv 0\!\!\pmod{2}.
 \end{cases}
\end{equation*}
If $k_2\equiv 1\!\pmod{2}$, then $(k_1+1)k_3 \equiv 0\!\pmod{2}$, and so 
\[ \inv(\psi_2(\tau))=(k_1+1)k_3+k_2\equiv 1\!\pmod{2}. \] 
On the other hand, if $k_2\equiv 0\!\pmod{2}$, then $k_1k_3\equiv 1\!\pmod{2}$, and so 
\[ \inv(\psi_2(\tau))=(k_1+2)k_3+k_2\equiv k_1k_3\equiv 1\!\pmod{2}. \]
In conclusion, $\psi_2(\tau)$ is odd, avoids $1342$, and it contains a $1423$ pattern whenever $\tau$ contains a $1342$ pattern. As for $\psi_1$, it can be easily verified that $\psi_2$ is bijective.
\end{proof}

%%%%%%%%%%%%%%%%%%%%%%%%%%%%%%%%%%%%%%%%%%%%
\section{Bijections to integer partitions and multidigraphs}
\label{sec:bijections}
In this section, we discuss combinatorial interpretations for the elements of $\Gr^{odd}_{n}(2413)$, enumerated by the sequence \cite[A006918]{Sloane}, and for the elements of $\Gr^{odd}_{n}(3412)$, counted by the sequence \cite[A005993]{Sloane}.

It is not hard to see that every permutation $\tau\in \Gr^*_{n}(2413)$ may be written as an inflation
\[ \tau = 1324[\id_{k_1},\id_{k_2},\id_{k_3},\id_{k_4}], \]
where $k_1,k_4\ge 0$, $k_2,k_3\ge 1$, $\sum k_j=n$, and $\id_0=\varepsilon$. Moreover, $\inv(\tau)=k_2\cdot k_3$, hence the permutation $\tau$ is odd if and only if $k_2\equiv k_3\equiv 1\!\pmod{2}$. The map 
$\tau\mapsto (k_1,k_2,k_3,k_4)$ gives a bijection between $\Gr^*_{n}(2413)$ and the set of weak compositions of $n$ where parts 2 and 3 are greater than zero.

\begin{prop}
The elements of $\Gr^{odd}_{n}(2413)$ are in one-to-one correspondence with the integer partitions of $n+2$ having Durfee square of size $2$. 
\end{prop}
\begin{proof}
For $\tau\in \Gr^{odd}_{n}(2413)$ let $w_\tau=(k_{1},k_{2},k_{3},k_{4})$ be its corresponding weak composition of $n$. Recall that $k_2$ and $k_3$ must both be odd numbers. We then create the Ferrers diagram of a partition of $n+2$ through the following process. To the diagram of the partition $2+2$, we append circles vertically downward or horizontally to the right based on the four elements of $w_\tau$ as shown:
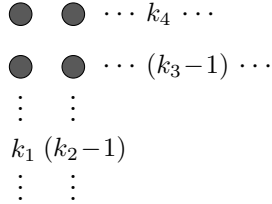
\begin{center}
\def\rad{0.15}
\begin{tikzpicture}[scale=1]
\draw[fill=black!65] foreach \i in {1,2}{ (0.7*\i,0.7) circle (\rad)};
\draw[fill=black!65] foreach \i in {1,2}{ (0.7*\i,0) circle (\rad)};
\draw (0.7,-0.45) node {$\vdots$};
\draw (0.74,-1.1) node {\small $k_{1}$};
\draw (0.7,-1.5) node {$\vdots$};
\draw (1.4,-0.45) node {$\vdots$};
\draw (1.55,-1.1) node {\small $(k_2\!-\!1)$};
\draw (1.4,-1.5) node {$\vdots$};
\draw (2.95,0) node {$\cdots$\,{\small $(k_3\!-\!1)$} $\cdots$};
\draw (2.57,0.7) node {$\cdots$\,{\small $k_{4}$} $\cdots$};
\end{tikzpicture}    
\end{center}
This representation is not necessarily that of an integer partition, so we make adjustments to the first two columns (and similarly to the first two rows) as follows:
\begin{enumerate}[\;(i)]
  \item $k_{1} \equiv 1\!\pmod{2}$ and $k_{1} > k_2-1$: do nothing.
  \item $k_{1} \equiv 1\!\pmod{2}$ and $k_{1} < k_{2}-1$: swap the two columns.
  \item $k_{1} \equiv 0\!\pmod{2}$ and $k_{1} > k_{2}-1$: move one circle from the first column to the second.
  \item $k_{1} \equiv 0\!\pmod{2}$ and $k_{1} \leq k_{2}-1$: swap the two columns.
\end{enumerate}
We then apply the same adjustment to the first two rows of the diagram, with $k_4$ and $k_3$ playing the roles of $k_1$ and $k_2$, respectively, and we end up with a proper partition of $n+2$ with a Durfee square of size 2. For example, the permutation $\tau=1263457$ corresponds to the weak composition $w_\tau=(2,1,3,1)$ and leads to
\medskip
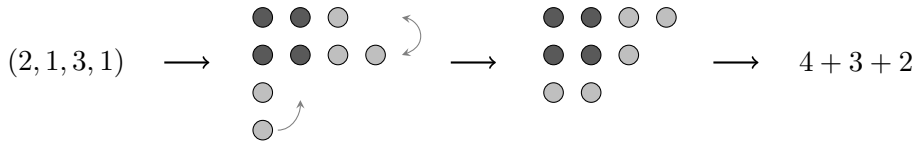
\begin{center}
\def\rad{0.13}
\begin{tikzpicture}[scale=1]
\node at (0,-0.6) {$(2,1,3,1)$};
\draw[thick,->] (1.3,-0.6) -- (1.9,-0.6);
\draw[thick,->] (8.6,-0.6) -- (9.2,-0.6);
\node at (10.5,-0.6) {$4+3+2$};
\begin{scope}[xshift=60]
\draw[fill=black!65] foreach \i in {1,2}{ (0.5*\i,0) circle (\rad)};
\draw[fill=black!65] foreach \i in {1,2}{ (0.5*\i,-0.5) circle (\rad)};
\draw[fill=gray!50] foreach \i in {2,3}{ (0.5,-0.5*\i) circle (\rad)};
\draw[fill=gray!50] foreach \i in {3,4}{ (0.5*\i,-0.5) circle (\rad)};
\draw[fill=gray!50] foreach \i in {3}{ (0.5*\i,0) circle (\rad)};
\draw[-stealth,gray] (0.7,-1.5) to[out=0,in=-90] (1,-1.1);
\draw[stealth-stealth,gray] (2.35,-0.5) to[out=20,in=-90] (2.6,-0.25) to[out=90,in=-10] (2.35,0);
\draw[thick,->] (3,-0.6) -- (3.6,-0.6);
\end{scope} 
\begin{scope}[xshift=170]
\draw[fill=black!65] foreach \i in {1,2}{ (0.5*\i,0) circle (\rad)};
\draw[fill=black!65] foreach \i in {1,2}{ (0.5*\i,-0.5) circle (\rad)};
\draw[fill=gray!50] foreach \i in {2}{ (0.5,-0.5*\i) circle (\rad)};
\draw[fill=gray!50] foreach \i in {2}{ (1,-0.5*\i) circle (\rad)};
\draw[fill=gray!50] foreach \i in {3}{ (0.5*\i,-0.5) circle (\rad)};
\draw[fill=gray!50] foreach \i in {3,4}{ (0.5*\i,0) circle (\rad)};
\end{scope}
\end{tikzpicture} 
\end{center}

The map is reversible since all cases are identifiable in the Ferrers diagram of the partition based on the parities of the first two columns and first two rows. The four procedures (i)--(iv) lead to the parities odd-even, even-odd, odd-odd, and even-even, respectively.
\end{proof}

\medskip
We end this section with an interesting interpretation for the elements of $\Gr^{odd}_{n}(3412)$.

A {\em multidigraph} is a directed graph that is allowed to have multiple edges (arcs and loops). For $n\ge 0$, let $\V_n$ be the set of multidigraphs on two nodes having $n$ edges. The elements of $\V_n$ for $n=0,1,2$ are illustrated in Table~\ref{tab:V_n_Examples}.
\begin{table}[ht]
\begin{tabular}{l||cl}
$n=0$\hspace{1ex}
&& 
\parbox[m]{.5\textwidth}{
\begin{tikzpicture}[scale=0.5]
\draw[fill] (0,0) circle (0.15);
\draw[fill] (2,0) circle (0.15);
\end{tikzpicture}
}
\\[6pt]\hline
$n=1$
&& 
\parbox[m]{.5\textwidth}{
\begin{tikzpicture}[scale=0.5, decoration={markings, mark= at position 0.6 with {\arr}}]
\begin{scope}
\draw[fill] (0,0) circle (0.15);
\draw[fill] (2,0) circle (0.15);
\draw[thick, scale=2.5] (0.8,0)  to[in=50,out=-50,loop] (0.8,0);
\end{scope}
\begin{scope}[xshift=160]
\draw[fill] (0,0) circle (0.15);
\draw[fill] (2,0) circle (0.15);
\draw[thick,postaction={decorate}] (0,0) -- (2,0);
\end{scope}
\end{tikzpicture}
}
\\ \hline
$n=2$ 
&&
\parbox[m]{.5\textwidth}{
\begin{tikzpicture}[scale=0.5,decoration={markings,mark= at position 0.6 with {\arr}}]
\draw[fill] (0,0) circle (0.15);
\draw[fill] (2,0) circle (0.15);
\draw[thick,postaction={decorate}] (2,0) to [in=35, out=145] (0,0);
\draw[thick,postaction={decorate}] (2,0) to [in=-35, out=-145] (0,0);
\begin{scope}[xshift=160]
\draw[fill] (0,0) circle (0.15);
\draw[fill] (2,0) circle (0.15);
\draw[thick,postaction={decorate}] (0,0) to [in=145, out=35] (2,0);
\draw[thick,postaction={decorate}] (2,0) to [in=-35, out=-145] (0,0);
\end{scope}
\begin{scope}[xshift=320]
\draw[fill] (0,0) circle (0.15);
\draw[fill] (2,0) circle (0.15);
\draw[thick, scale=2.5] (0.8,0) to[in=50,out=-50,loop] (0.8,0);
\draw[thick,postaction={decorate}] (0,0) -- (2,0);
\end{scope}
\end{tikzpicture}
}
\\[-6pt]
&& 
\begin{tikzpicture}[scale=0.5,decoration={markings,mark= at position 0.6 with {\arr}}]
\begin{scope}
\draw[fill] (0,0) circle (0.15);
\draw[fill] (2,0) circle (0.15);
\draw[thick, scale=2.5] (0.8,0) to[in=50, out=-50, loop] (0.8,0);
\draw[thick,postaction={decorate}] (2,0) -- (0,0);
\end{scope}
\begin{scope}[xshift=160]
\draw[fill] (0,0) circle (0.15);
\draw[fill] (2,0) circle (0.15);
\draw[thick, scale=2.5] (0.8,0) to[in=50, out=-50, loop] (0.8,0);
\draw[thick, scale=2.5] (0,0) to[in=130,out=-130,loop] (0,0);
\end{scope}
\begin{scope}[xshift=320]
\draw[fill] (0,0) circle (0.15);
\draw[fill] (2,0) circle (0.15);
\draw[thick, scale=2.5] (0.8,0) to[in=110,out=10,loop] (0.8,0);
\draw[thick, scale=2.5] (0.8,0) to[in=-110,out=-10,loop] (0.8,0);
\end{scope}
\end{tikzpicture} 
\end{tabular}
\bigskip
\caption{Elements of $\V_0$, $\V_1$, and $\V_2$.} \label{tab:V_n_Examples}
\end{table}

If we think of the two nodes as `left node' and `right node', then every $G\in \V_n$ can be identified with a weak composition $\hat t_G = (t_1,t_2,t_3,t_4)\in \W_4(n)$, where  
\begin{itemize}
\item[] $t_{1}=$ number of loops on the left node,
\item[] $t_{2}=$ number of edges that go from the left node to the right,
\item[] $t_{3}=$ number of edges that go from the right node to the left,
\item[] $t_{4}=$ number of loops on the right node.
\end{itemize}
However, since the elements of $\V_n$ are invariant under rotation, $G$ could also be represented by the weak composition $\hat t'_G = (t_4,t_3,t_2,t_1)$. Therefore, we define the equivalence relation
\[ (t_1,t_2,t_3,t_4) \sim (t_4,t_3,t_2,t_1), \]
and let $\P_{n} = \W_4(n)/\!\sim$ be the set of equivalence classes. For instance, 
\[ \P_0 = \{[(0,0,0,0)]\} \text{ and } \P_1 = \{[(1,0,0,0)],[(0,1,0,0)]\}. \]
By definition, the elements of $\V_n$ are in one-to-one correspondence with those of $\P_{n}$.

\begin{prop}\label{prop:oddtoP}
For $n\ge 2$, the set $\Gr^{odd}_{n}(3412)$ is in bijection with the set $\P_{n-2}$.
\end{prop}

\begin{proof}
For $u=(u_1,u_2,u_3,u_4)$, we let $u'=(u_4,u_3,u_2,u_1)$ and write $u\prec u'$ if either $u_1<u_4$ or $u_1=u_4$ and $u_2<u_3$. We say that $u$ is {\em minimal} if $u\prec u'$. Clearly, if $u\not= u'$, then $u$ is minimal if and only if $u'$ is not. An element $u$ is called {\em symmetric} if $u = u'$.

We will use the maps $\Lambda$ and $\phi$ introduced in the proof of Theorem~\ref{thm:3412oddGrass} to define a bijective map $\xi:\Gr^{odd}_{n}(3412)\to \P_{n-2}$. Recall that $\phi$ is only defined on the sets $\Lambda\big(\Gr^{odd}_{2m-1}(3412)\big)$ and $\Lambda\big(\Gr^{odd}_{2m}(3412)\backslash \A_{2m}\big)$. Also recall that the elements of $\A_{2m}$ are of the form $\tau = \tau_2\, (2m) 1\, \tau_3$ with $|\tau_2|\equiv|\tau_3|\equiv 0\!\pmod{2}$, so $\Lambda(\tau)=(0,|\tau_2|,|\tau_3|,0)$.

Let $\tau\in \Gr^{odd}_{n}(3412)$ and $\hat t = \Lambda(\tau)$. If $\hat t$ is minimal or if $\hat t'$ is minimal and $\phi(\hat t)$ is defined, we let
\begin{equation*}
 \xi(\tau) = \begin{cases}
 [\,\hat t\,] &\text{if $\hat t$ is minimal}, \\
 [\phi(\hat t)] &\text{if $\hat t'$ is minimal}.
 \end{cases}
\end{equation*}
Otherwise, either $\hat t$ is symmetric or $\hat t'$ is minimal but $\phi(\hat t)$ is undefined (which means $\tau\in\A_n$). In both cases, $n$ must be even and $t_1+t_2+t_3+t_4=n-2$.

If $n\equiv 0\!\pmod4$ and $\hat t$ is symmetric, say $\hat t=(t_1,t_2,t_2,t_1)$, then we let
\begin{equation*}
 \xi(\tau) = \begin{cases}
 [\,\hat t\,] &\text{if $t_1$ is even}, \\
 [\phi(\hat t)] &\text{if $t_1$ is odd}.
 \end{cases}
\end{equation*}
On the other hand, if $\tau\in\A_n$ and $\hat t=\Lambda(\tau)$ is not minimal, then $\hat t = (0,t_2,t_3,0)$ with $t_2>t_3$ and $t_2\equiv t_3\equiv 0\!\pmod2$. In this case, we let
\begin{equation*}
 \xi(\tau) = \begin{cases}
 \left[\left(\tfrac{t_2}{2},\tfrac{t_3}{2},\tfrac{t_3}{2},\tfrac{t_2}{2}\right)\right] &\text{if $t_2/2$ is odd}, \\
 \left[\left(\tfrac{t_3}{2},\tfrac{t_2}{2},\tfrac{t_2}{2},\tfrac{t_3}{2}\right)\right] &\text{if $t_2/2$ is even}.
 \end{cases}
\end{equation*}

If $n\equiv 2\!\pmod4$ and $\hat t=(t_1,t_2,t_2,t_1)$ (i.e.\ symmetric), we let
\begin{equation*}
 \xi(\tau) = \begin{cases}
 [\,\hat t\,] &\text{if $t_1$ is odd or $t_1=0$}, \\
 [\phi(\hat t)] &\text{if $t_1$ is even, $t_1>0$}.
 \end{cases}
\end{equation*}
Finally, if $\tau\in\A_n$ and $\hat t=\Lambda(\tau)$ is neither minimal nor symmetric, then we let
\begin{equation*}
 \xi(\tau) = \begin{cases}
 \left[\left(\tfrac{t_2}{2},\tfrac{t_3}{2},\tfrac{t_3}{2},\tfrac{t_2}{2}\right)\right] &\text{if $t_2/2$ is even}, \\
 \left[\left(\tfrac{t_3}{2}+1,\tfrac{t_2}{2}-1,\tfrac{t_2}{2}-1,\tfrac{t_3}{2}+1\right)\right] &\text{if $t_2/2$ is odd}.
 \end{cases}
\end{equation*}

For every $n$, the map $\xi$ sends the elements of $\Gr^{odd}_{n}(3412)$ into three distinguishable disjoint subsets of $\P_{n-2}$ whose elements can be recognized by their symmetry properties and the parities of their components. Observe that if $\tau$ is odd and $\hat t=\Lambda(\tau)$, then $t_2\equiv t_3\pmod{2}$ and therefore $\phi(\hat t)_2\not\equiv \phi(\hat t)_3\pmod{2}$.

The inverse map $\xi^{-1}:\P_{n-2}\to \Gr^{odd}_{n}(3412)$ is defined as follows. Let $[u] = [(u_1,u_2,u_3,u_4)]$ be an equivalence class in $\P_{n-2}$. 
\begin{enumerate}[\quad(i)]
  \item If $u$ is not symmetric and $u_2 \equiv u_3\!\pmod{2}$, choose the representative of $[u]$ that is minimal, call it $\hat t_{\min}$, and let 
  \[  \xi^{-1}([u]) = \Lambda^{-1}(\hat t_{\min}). \]
  \item If $u_2 \not\equiv u_3$, consider $\phi^{-1}(u)$ and $\phi^{-1}(u')$. By Lemma~\ref{lem:symmetric}, they both have the same symmetry properties. If they are both not symmetric, Lemma~\ref{lem:maxmin} implies that one and only one of them must be minimal, say $\hat t_{\min} = \phi^{-1}(u)$. We then denote $\hat t_{\max} = \phi^{-1}(u')$ and let
  \[  \xi^{-1}([u]) = \Lambda^{-1}(\hat t_{\max}). \]
\smallskip
If $\phi^{-1}(u)$ and $\phi^{-1}(u')$ are both symmetric and $u\prec u'$, then
  \[  \xi^{-1}([u]) = \Lambda^{-1}(\phi^{-1}(u')). \]
\item If $u$ is symmetric and $\phi(u)$ is minimal, then
  \[  \xi^{-1}([u]) = \Lambda^{-1}(u). \]
If $u$ is symmetric and $\phi(u)$ is not minimal (or undefined), we consider two cases:
\begin{enumerate}[\quad(a)]
 \item If $n\equiv0\!\pmod4$, then $u = (2k+1,2\ell,2\ell,2k+1)$. We then let
 \begin{equation*}
 \xi^{-1}([u]) = \begin{cases}
 \Lambda^{-1}(0,4k+2,4\ell,0) &\text{ if } k\ge\ell, \\
 \Lambda^{-1}(0,4\ell,4k+2,0) &\text{ if } k<\ell.
 \end{cases}
 \end{equation*} 
 \item If $n\equiv2\!\pmod4$, then $u = (2k,2\ell,2\ell,2k)$. In this case, we let
 \begin{equation*}
 \xi^{-1}([u]) = \begin{cases}
 \Lambda^{-1}(0,4k,4\ell,0) &\text{ if } k>\ell, \\
 \Lambda^{-1}(0,4\ell+2,4k-2,0) &\text{ if } 0<k\le \ell, \\
 \Lambda^{-1}(0,2\ell,2\ell,0) &\text{ if } k=0.
 \end{cases}
 \end{equation*} 
\end{enumerate}
\end{enumerate}
\end{proof}
The bijective map $\xi$ is illustrated in Table~\ref{tab:xi_mapExample} for $n=6$. The elements in blue (marked with $*$) are the ones obtained through the map $\phi$. The element in red (marked with $\star$) is the one obtained through the process for the permutations $\tau\in\A_n$ for which $\Lambda(\tau)$ is neither minimal nor symmetric.

\begin{table}[ht]
\renewcommand{\arraystretch}{1.3}
\begin{tabular}{c | c | c | l}
$\tau\in \Gr_{6}^{odd}(3412)$ &  $\hat t=\Lambda(\tau)$ & Property & \;$\xi(\tau)\in \P_{4}$\; \\
\hline\hline
\R $1\, 2\, 3\, 4\, 6\, 5$ & $(4, 0, 0, 0)$ & $\hat t'$ minimal & \color{blue} $[(3,1,0,0)]*$\\ 
$1\, 2\, 3\, 5\, 4\, 6$ & $(3, 0, 0, 1)$ & $\hat t'$ minimal & \color{blue} $[(3,0,1,0)]*$\\ 
$1\, 2\, 4\, 3\, 5\, 6$ & $(2, 0, 0, 2)$ & symm., $t_1>0$ even & \color{blue} $[(2,0,1,1)]*$\\ 
$1\, 2\, 4\, 5\, 6\, 3$ & $(2, 2, 0, 0)$ & $\hat t'$ minimal & \color{blue} $[(1,3,0,0)]*$\\ 
$1\, 2\, 4\, 6\, 3\, 5$ & $(2, 1, 1, 0)$ & $\hat t'$ minimal & \color{blue} $[(2,1,0,1)]*$\\ 
$1\, 2\, 6\, 3\, 4\, 5$ & $(2, 0, 2, 0)$ & $\hat t'$ minimal &\color{blue} $[(1,1,2,0)]*$\\ 
$1\, 3\, 2\, 4\, 5\, 6$ & $(1, 0, 0, 3)$ & minimal & $[(1,0,0,3)]$\\ 
$1\, 3\, 4\, 5\, 2\, 6$ & $(1, 2, 0, 1)$ & $\hat t'$ minimal &\color{blue}$[(1,2,1,0)]*$\\ 
$1\, 3\, 5\, 2\, 4\, 6$ & $(1, 1, 1, 1)$ & symm., $t_1$ odd & $[(1, 1, 1, 1)]$\\ 
$1\, 5\, 2\, 3\, 4\, 6$ & $(1, 0, 2, 1)$ & minimal & $[(1, 0, 2, 1)]$\\ 
$2\, 1\, 3\, 4\, 5\, 6$ & $(0, 0, 0, 4)$ & minimal & $[(0, 0, 0, 4)]$\\ 
$2\, 3\, 4\, 1\, 5\, 6$ & $(0, 2, 0, 2)$ & minimal & $[(0, 2, 0, 2)]$\\ 
$2\, 3\, 4\, 5\, 6\, 1$ & $(0, 4, 0, 0)$ & not min., $\tau\in\A_6$ &\color{black!30!red} $[(2,0,0,2)]\star$\\ 
$2\, 3\, 4\, 6\, 1\, 5$ & $(0, 3, 1, 0)$ & $\hat t'$ minimal &\color{blue} $[(0,3,0,1)]*$\\ 
$2\, 3\, 6\, 1\, 4\, 5$ & $(0, 2, 2, 0)$ & symm., $t_1=0$ & $[(0, 2, 2, 0)]$\\ 
$2\, 4\, 1\, 3\, 5\, 6$ & $(0, 1, 1, 2)$ & minimal & $[(0, 1, 1, 2)]$\\ 
$2\, 6\, 1\, 3\, 4\, 5$ & $(0, 1, 3, 0)$ & minimal & $[(0, 1, 3, 0)]$\\ 
$4\, 1\, 2\, 3\, 5\, 6$ & $(0, 0, 2, 2)$ & minimal & $[(0, 0, 2, 2)]$\\ 
$6\, 1\, 2\, 3\, 4\, 5$ & $(0, 0, 4, 0)$ & minimal & $[(0, 0, 4, 0)]$
\end{tabular}
\medskip
\caption{Bijective map $\xi:\Gr_{6}^{odd}(3412)\to \P_4$.} \label{tab:xi_mapExample}
\end{table}

\begin{lem}\label{lem:symmetric}
Let $u,v\in \W_4(n-2)$. If $\phi(u)\sim\phi(v)$ and $u$ is symmetric, then $v$ is symmetric.
\end{lem}
\begin{proof}
Suppose $u_2 = u_3 \equiv 0\!\pmod2$. If $n\equiv0\pmod4$, then $u = (2k+1,2\ell,2\ell,2k+1)$ for some $k,\ell\geq0$. So, $\phi(u) = (2k+1,2\ell,2\ell+1,2k)$ and $\phi(v) = (2k,2\ell+1,2\ell,2k+1)$. Therefore, $v = (2k,2\ell+1,2\ell+1,2k)$. If $n\equiv2\pmod4$, $u$ can be written as $u = (2k,2\ell,2\ell,2k)$ where $k\geq1$ and $\ell\geq0$, leading to $v = (2k-1,2\ell+1,2\ell+1,2k-1)$. The argument when $u_2 = u_3\equiv 1\!\pmod2$ is similar and can be seen as going from $v$ to $u$.
\end{proof}

\begin{lem}\label{lem:maxmin}
Let $u,v\in \W_4(n-2)$, not symmetric. If $\phi(u)\sim\phi(v)$ and $u$ is not minimal, then $v$ is minimal. Similarly, if $\phi(u)\sim\phi(v)$ and $u$ is minimal, then $v$ is not minimal.
\end{lem}

\begin{proof}
Let $\phi(u) = (u_1,u_2,u_3,u_4)$ and $\phi(v) = (u_4,u_3,u_2,u_1)$. Since $u_2 \not\equiv u_3\!\pmod 2$, there are three cases to consider, $u_3\equiv1$, $u_3\equiv0$ with $u_4>0$, and $u_3\equiv0$ with $u_4=0$. If $u_3\equiv1$, then $u = (u_1,u_2,u_3-1,u_4+1)$ and, for $u$ not to be minimal, we must have either $u_1>u_4+1$ or $u_1=u_4+1$ and $u_2>u_3-1$. This guarantees $u_1>1$, so $v = (u_4,u_3,u_2+1,u_1-1)$. Thus, $v$ is minimal. The other cases are similar and follow directly from the definitions.
\end{proof}

\begin{cor}
The elements of $\Gr^{odd}_{n}(3412)$ are in one-to-one correspondence with the multidigraphs on two nodes having $n-2$ edges.
\end{cor}
For $n=5$, the above correspondence is illustrated in Figure~\ref{fig:multidigraphs}.

\begin{figure}[ht]
\begin{tikzpicture}[scale=0.7,decoration={markings,mark= at position 0.55 with {\arr}}]
\begin{scope}
\draw[fill] (0,0) circle (0.12);
\draw[fill] (2,0) circle (0.12);
\draw[thick,postaction={decorate}] (0,0) to [in=145, out=35] (2,0);
\draw[thick,postaction={decorate}] (0,0) to [in=-145, out=-35] (2,0);
\draw[thick,postaction={decorate}] (0,0) -- (2,0);
\node at (1.1,-1.3) {\small\color{blue} $[(0,3,0,0)]*$};
\node at (1,-2.1) {1 3 4 5 2};
\end{scope}
\begin{scope}[xshift=115]
\draw[fill] (0,0) circle (0.12);
\draw[fill] (2,0) circle (0.12);
\draw[thick,postaction={decorate}] (2,0) to [in=35, out=145] (0,0);
\draw[thick,postaction={decorate}] (2,0) to [in=-35, out=-145] (0,0);
\draw[thick,postaction={decorate}] (0,0) -- (2,0);
\node at (1.1,-1.3) {\small\color{blue} $[(0,1,2,0)]*$};
\node at (1,-2.1) {1 5 2 3 4};
\end{scope}
\begin{scope}[xshift=230]
\draw[fill] (0,0) circle (0.12);
\draw[fill] (2,0) circle (0.12);
\draw[thick,postaction={decorate}] (2,0) -- (0,0);
\draw[thick, scale=2] (0,0)  to[in=85,out=175,loop] (0,0);
\draw[thick, scale=2] (0,0)  to[in=185,out=275,loop] (0,0);
\node at (1.1,-1.3) {\small\color{blue} $[(2,0,1,0)]*$};
\node at (1,-2.1) {1 2 4 3 5};
\end{scope}
\begin{scope}[xshift=345]
\draw[fill] (0,0) circle (0.12);
\draw[fill] (2,0) circle (0.12);
\draw[thick,postaction={decorate}] (0,0) -- (2,0);
\draw[thick, scale=2] (0,0)  to[in=85,out=175,loop] (0,0);
\draw[thick, scale=2] (0,0)  to[in=185,out=275,loop] (0,0);
\node at (1.1,-1.3) {\small\color{blue} $[(2,1,0,0)]*$};
\node at (1,-2.1) {1 2 3 5 4};
\end{scope}
\begin{scope}[xshift=460]
\draw[fill] (0,0) circle (0.12);
\draw[fill] (2,0) circle (0.12);
\draw[thick,postaction={decorate}] (0,0) -- (2,0);
\draw[thick, scale=2] (0,0)  to[in=130,out=-130,loop] (0,0);
\draw[thick, scale=2] (1,0)  to[in=50,out=-50,loop] (1,0);
\node at (1.1,-1.3) {\small\color{blue} $[(1,1,0,1)]*$};
\node at (1,-2.1) {1 3 5 2 4};
\end{scope}
\end{tikzpicture}

\smallskip

\begin{tikzpicture}[scale=0.7,decoration={markings,mark= at position 0.55 with {\arr}}]
\begin{scope}
\draw[fill] (0,0) circle (0.12);
\draw[fill] (2,0) circle (0.12);
\draw[thick,postaction={decorate}] (2,0) to [in=35, out=145] (0,0);
\draw[thick,postaction={decorate}] (2,0) to [in=-35, out=-145] (0,0);
\draw[thick, scale=2] (1,0)  to[in=50,out=-50,loop] (1,0);
\node at (1,-1.3) {\small $[(0,0,2,1)]$};
\node at (1,-2.1) {4 1 2 3 5};
\end{scope}
\begin{scope}[xshift=115]
\draw[fill] (0,0) circle (0.12);
\draw[fill] (2,0) circle (0.12);
\draw[thick,postaction={decorate}] (0,0) to [in=145, out=35] (2,0);
\draw[thick,postaction={decorate}] (2,0) to [in=-35, out=-145] (0,0);
\draw[thick, scale=2] (1,0)  to[in=50,out=-50,loop] (1,0);
\node at (1,-1.3) {\small $[(0,1,1,1)]$};
\node at (1,-2.1) {2 4 1 3 5};
\end{scope}
\begin{scope}[xshift=230]
\draw[fill] (0,0) circle (0.12);
\draw[fill] (2,0) circle (0.12);
\draw[thick,postaction={decorate}] (0,0) to [in=145, out=35] (2,0);
\draw[thick,postaction={decorate}] (0,0) to [in=-145, out=-35] (2,0);
\draw[thick, scale=2] (1,0)  to[in=50,out=-50,loop] (1,0);
\node at (1,-1.3) {\small $[(0,2,0,1)]$};
\node at (1,-2.1) {2 3 4 1 5};
\end{scope}
\begin{scope}[xshift=345]
\draw[fill] (0,0) circle (0.12);
\draw[fill] (2,0) circle (0.12);
\draw[thick, scale=2] (0,0)  to[in=130,out=-130,loop] (0,0);
\draw[thick, scale=2] (1,0)  to[in=95,out=5,loop] (1,0);
\draw[thick, scale=2] (1,0)  to[in=-95,out=-5,loop] (1,0);
\node at (1,-1.3) {\small $[(1,0,0,2)]$};
\node at (1,-2.1) {1 3 2 4 5};
\end{scope}
\begin{scope}[xshift=460]
\draw[fill] (0,0) circle (0.12);
\draw[fill] (2,0) circle (0.12);
\draw[thick, scale=2] (1,0)  to[in=45,out=135,loop] (1,0);
\draw[thick, scale=2] (1,0)  to[in=-45,out=-135,loop] (1,0);
\draw[thick, scale=2] (1,0)  to[in=45,out=-45,loop] (1,0);
\node at (1,-1.3) {\small $[(0,0,0,3)]$};
\node at (1,-2.1) {2 1 3 4 5};
\end{scope}
\end{tikzpicture}
\caption{Multidigraphs corresponding to the elements of $\Gr^{odd}_{5}(3412)$.}
\label{fig:multidigraphs}
\end{figure}
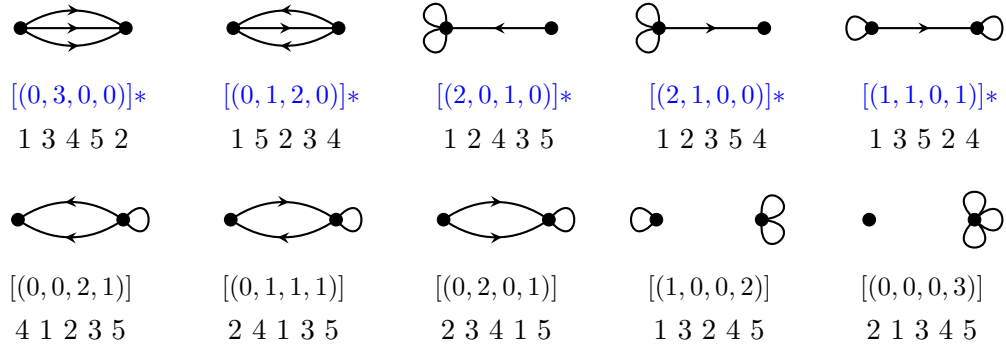

\section*{Acknowledgement}
We are very grateful to the Altoona Summer Undergraduate Research Fellowship that made it possible for Tomasko to dedicate herself to this project during the summer of 2021. 

%%%%%%%%%%%%%%%%%%%%%%%%%%%%%%%%%%%%%%%%%%%%

%%%%%%%%%%%%%%%%%%%%%%%%%%%%%%%%%%%%%%%%%%%%
\end{document}